\documentclass{amsproc}
\usepackage{graphicx,epsfig}

\def\N{{\mathbb N}}
\def\Z{{\mathbb Z}}
\def\R{{\mathbb R}}
\def\P{{\mathbb P}}

\def\S{{\mathcal S}}

\def\lin{{\rm lin}}
\def\Vol{{{\rm Vol}}} 
\def\vol{{\rm vol}}
\def\conv{{\rm conv}}
\def\deg{{\rm deg}}

\newtheorem{theorem}{Theorem}[section]

\theoremstyle{definition}

\newtheorem{exam}[theorem]{Example}

\newtheorem{coro}[theorem]{Corollary}
\newtheorem{conj}[theorem]{Conjecture}
\newtheorem{prop}[theorem]{Proposition}
\theoremstyle{remark}
\newtheorem{rem}[theorem]{Remark}

\numberwithin{equation}{section}

\newcommand{\abs}[1]{\lvert#1\rvert}

\begin{document}

\title{3-Dimensional Lattice Polytopes\\Without Interior Lattice Points}

\author{Jaron Treutlein}
\address{Department of Mathematics and Physics, 
University of T\"ubingen, Auf der Morgenstelle 10, D-72076 T\"ubingen, Germany}
\email{jaron@mail.mathematik.uni-tuebingen.de}





\begin{abstract}
A theorem of Howe states that every 3-dimensional lattice polytope 
$P$ whose only lattice points are its vertices, is a Cayley polytope, 
i.e. $P$ is the convex hull of two lattice polygons with distance one.\\
We want to generalize this result by classifying 3-dimensional lattice 
polytopes without interior lattice points. The main result will be, that 
they are up to finite many exceptions either Cayley polytopes or there 
is a projection, which maps the polytope to the double unimodular 
2-simplex.\\
To every such polytope we associate a smooth projective surface of 
genus $0$.
\end{abstract}

\maketitle

\section{Introduction}
\label{1}

Let $M\cong\Z^n$ be a lattice and $P\subset M\otimes_\Z\R$ be an 
$n$-dimensional lattice polytope, i.e. the set of vertices of $P$ 
is contained in $M$. However in the following we will allways 
suppose $M=\Z^n$.\\
Let $P\subset \R^3$ be a 3-dimensional lattice polytope. If 
$|P^\circ\cap\Z^3|=i>0$, then we will derive from Hensley's theorem 
\cite{He}, that its volume $\vol(P)$ is bounded by a constant 
depending only on $i$. Jeffrey Lagarias and G\"unter M. Ziegler 
proved in \cite{LZ}:
\begin{theorem}[Lagarias, Ziegler] 
 \label{LZ}
Let $F$ be a familiy of $n$-dimensional lattice polytopes. Then is 
equivalent:
\begin{enumerate}
 \item $\vol(P)<C$ $\forall P\in F$ with a constant $C>0$.
 \item Up to affine unimodular transformation, $F$ is a finite set.
\end{enumerate}
\end{theorem}
Thus the finiteness up to affine unimodular transformation of 
3-dimensional lattice polytopes having exactly $i>0$ interior lattice 
points follows. So, it remains to consider the case $i=0$ in order to 
classify all 3-dimensional lattice polytopes.

\medskip

Let $\{e_1,\ldots,e_n\}\subset\R^n$ be the standard basis of $\R^n$. 
We denote the unimodular $n$-simplex $\conv(0,e_1,\ldots,e_n)$ by 
$\Delta_n$. If two lattice polytopes $P$ and $Q$ are equivalent modulo 
affine unimodular transformation, we will describe it by $P\cong Q$.

\medskip

John R. Arkinstall \cite{Ar}, Askold Khovanskii \cite{Kh97}, Robert J. 
Koelman \cite{Ko} and Josef Schicho \cite{Sch} proved that a lattice 
polygon without interior lattice points satisfies $P\cong2\Delta_2$ or 
$P\cong\conv(0,e_1,h_1e_2,e_1+h_2e_2)$ with some heights $h_1,h_2\in\N$. 
The latter is called a Lawrence polygon. Thus a lattice polygon without 
interior lattice points either has a projection to $\Delta_1$, which is 
the only 1-dimensional lattice polytope without interior lattice points, 
or it is equivalent to $2\Delta_2$.

\medskip

Victor Batyrev conjected that any 3-dimensional lattice polytope without 
interior lattice points either can be projected to $\Delta_1$, to 
$2\Delta_2$ 
or its volume is bounded. From Theorem \ref{LZ} it will follow that up to 
unimodular transformation there is only a finite number of exceptional 
polytopes of this last class.

\medskip

A lattice polytope $P\subset\R^{n}$ is described as the Cayley polytope of 
the lattice polytopes $P_0,\ldots, P_{r-1}\subset\langle e_{r},\ldots,e_n
\rangle, 1<r\leq n,$ if $P\cong\conv(0\times P_0, e_1\times P_1,\ldots, 
e_{r-1}\times P_{r-1})$. Then we call $P$ simply a Cayley polytope and 
notice that these are the lattice polytopes which can be projected to 
$\Delta_1$.\\
An $n$-dimensional lattice polytope is called a $k$-fold lattice pyramid 
over an $(n-k)$-dimensional lattice polytope $Q$, if it is the Cayley 
polytope of $Q$ and $k$ lattice points.

\medskip

It is clear, that lattice polytopes having a projection to $\Delta_1$ or 
$2\Delta_2$, both have no interior lattice points.\\ Roger Howe proved in 
1977 the following:
\begin{theorem}[Howe]
\label{Howe}
Let $P\subset \R^3$ be a lattice polytope whose only lattice points are 
vertices. Then $|P\cap\Z^3|\leq8$ and $P$ is a Cayley polytope.
\end{theorem}
As he did not publish it, Herbert E. Scarf did it \cite{Sca}. There are 
further proofs for example by G.K. White \cite{Whi}, Bruce Reznick \cite{Re}, 
Andr\'{a}s Seb\H{o} \cite{Se}, David R. Morrison and Glenn Stevens \cite{MS}.

\medskip

The aim of this paper is to prove Victor Batyrev's conjecture and hence ge-
 neralize Theorem \ref{Howe}:
\begin{theorem}
\label{Main}
 Let $P\subset \R^3$ be a lattice polytope without interior lattice points. 
Then either is $P$ a Cayley polytope, $P$ can be projected to the double 
unimodular 2-simplex or $P$ is an exceptional polytope, whereas up to 
unimodular transformation there is only a finite number of these.
\end{theorem}

\medskip

Some of these exceptional simplices ($P_1,\ldots,P_6$) are in figure 1. 
All the others (there are 15 more simplices) are included in one of them. 
Thus we desribe $P_1,\ldots,P_6$ as maximal exceptional lattice simplices.

\medskip

 Moreover there are maximal exceptional lattice polytopes which are not 
simplices, see for example $P_7,P_8$ and $P_9$ in figure 2. It is not 
known if these are the only maximal lattice polytopes.\\
These are the known maximal exceptional lattice polytopes:
\begin{center}
$P_1:=\conv(0,e_1, 2e_1+5e_2, 3e_1+5e_3)$\\ 
$P_2:=\conv(0, 3e_1, e_1+3e_2, 2e_1+3e_3)$\\ 
$P_3:=\conv(0, 3e_1, 3e_2, 3e_3)$\\ 
$P_4:=\conv(0, 4e_1, 4e_2, 2e_3)$\\ 
$P_5:=\conv(0, 4e_1, 2e_1+4e_2, e_1+2e_3)$\\ 
$P_6:=\conv(0, 6e_1, 3e_2, 2e_3)$\\ 

\medskip
$P_7:=\conv(\pm 2e_1, \pm 2e_2, e_1+e_2+2e_3)$\\ 
$P_8:=\conv(\pm e_1, 2e_2, e_1+2e_3\pm e_1, e_1+2e_2+2e_3)$\\ 
$P_9:=\conv(\pm e_1, \pm e_2, e_1+e_2+2e_3 \pm 
e_1, e_1+e_2+2e_3\pm e_2).$ 
\end{center}

\begin{figure}[ht]
\begin{minipage}[c]{0.3\textwidth}
\centering \includegraphics[width=1.5in]{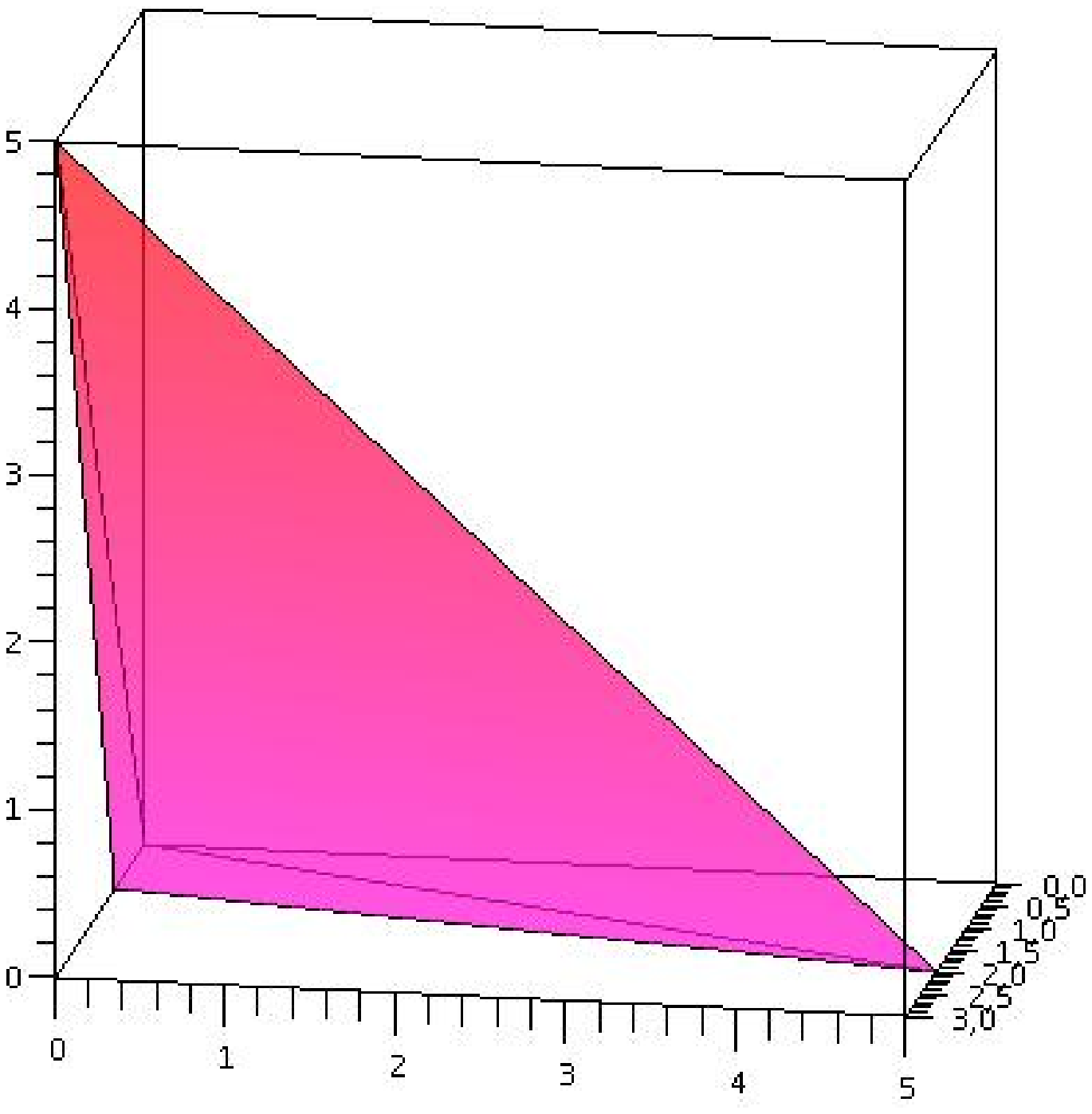} 
\end{minipage}
\begin{minipage}[c]{0.3\textwidth}
\centering \includegraphics[width=1.5in]{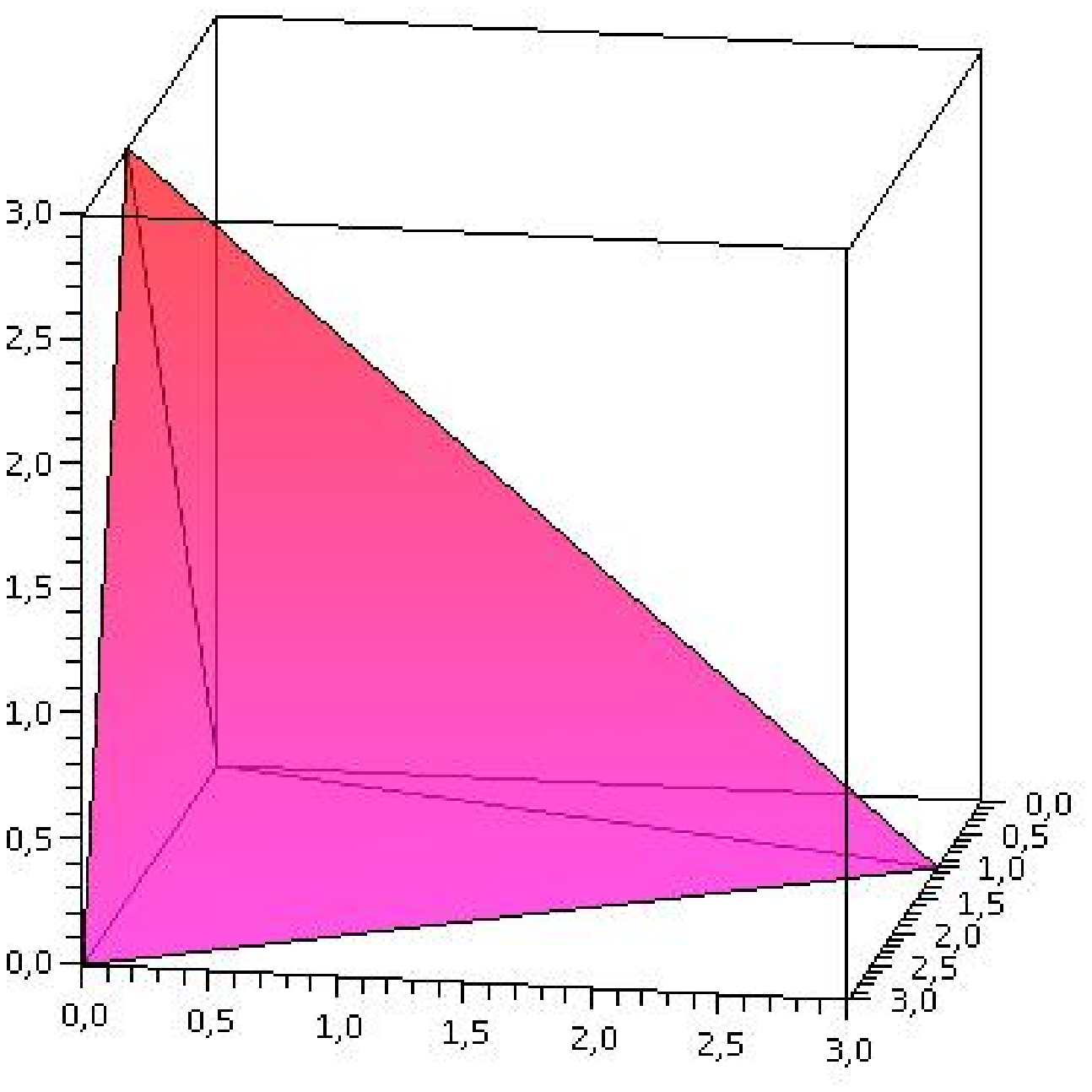} 
\end{minipage}
\begin{minipage}[c]{0.3\textwidth}
\centering\includegraphics[width=1.5in]{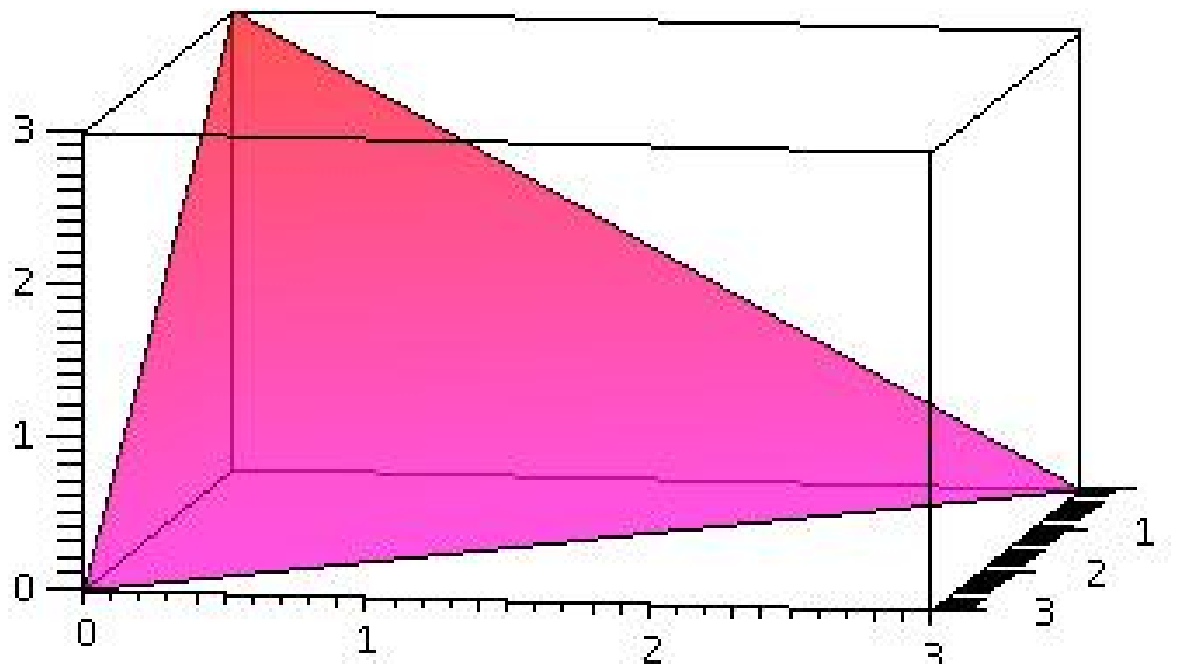} 
\end{minipage}
\begin{minipage}[c]{0.3\textwidth}
\centering\includegraphics[width=1.5in]{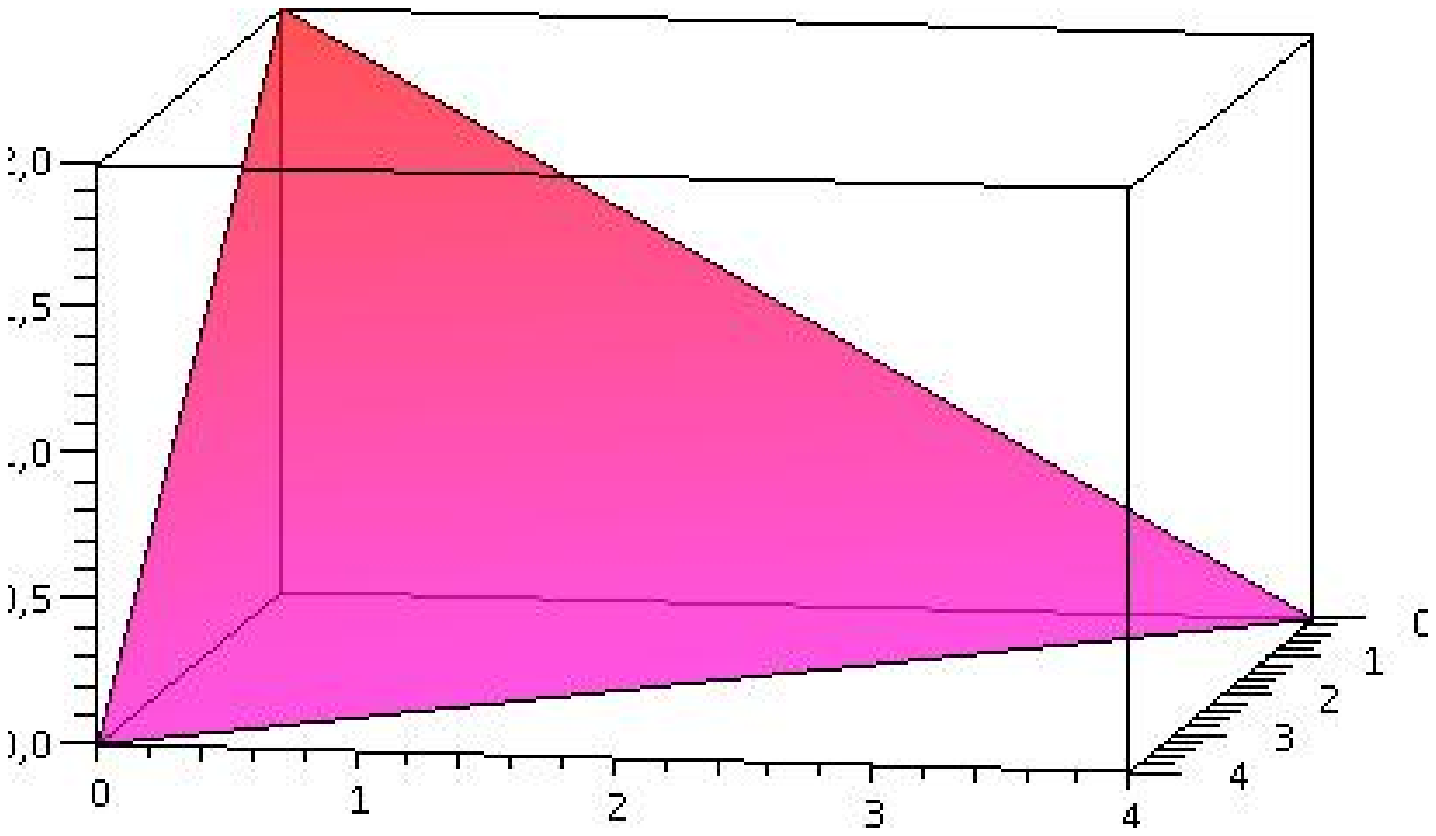} 
\end{minipage}
\begin{minipage}[c]{0.3\textwidth}
\centering\includegraphics[width=1.5in]{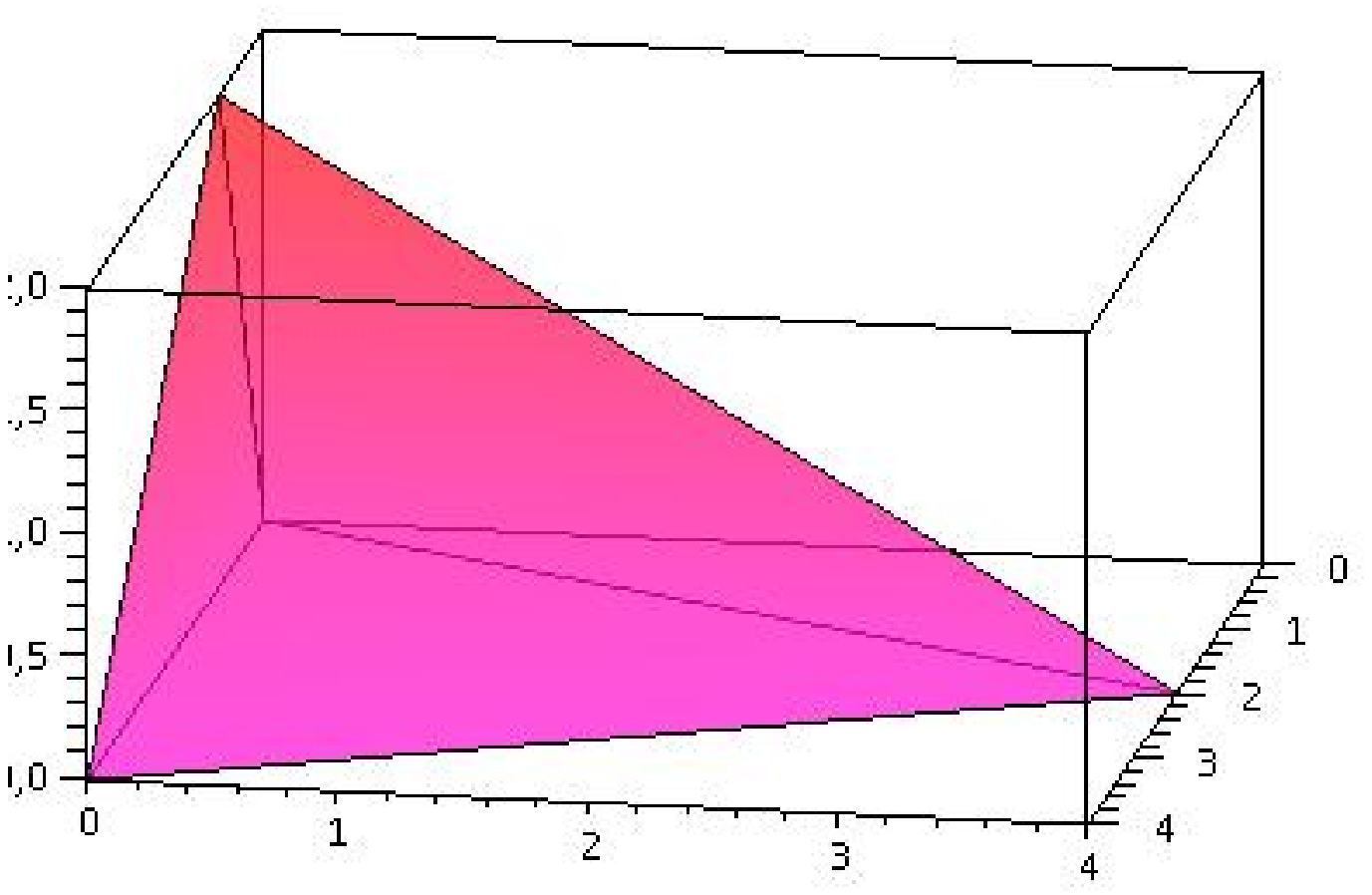} 
\end{minipage}
\begin{minipage}[c]{0.3\textwidth}
\centering\includegraphics[width=1.5in]{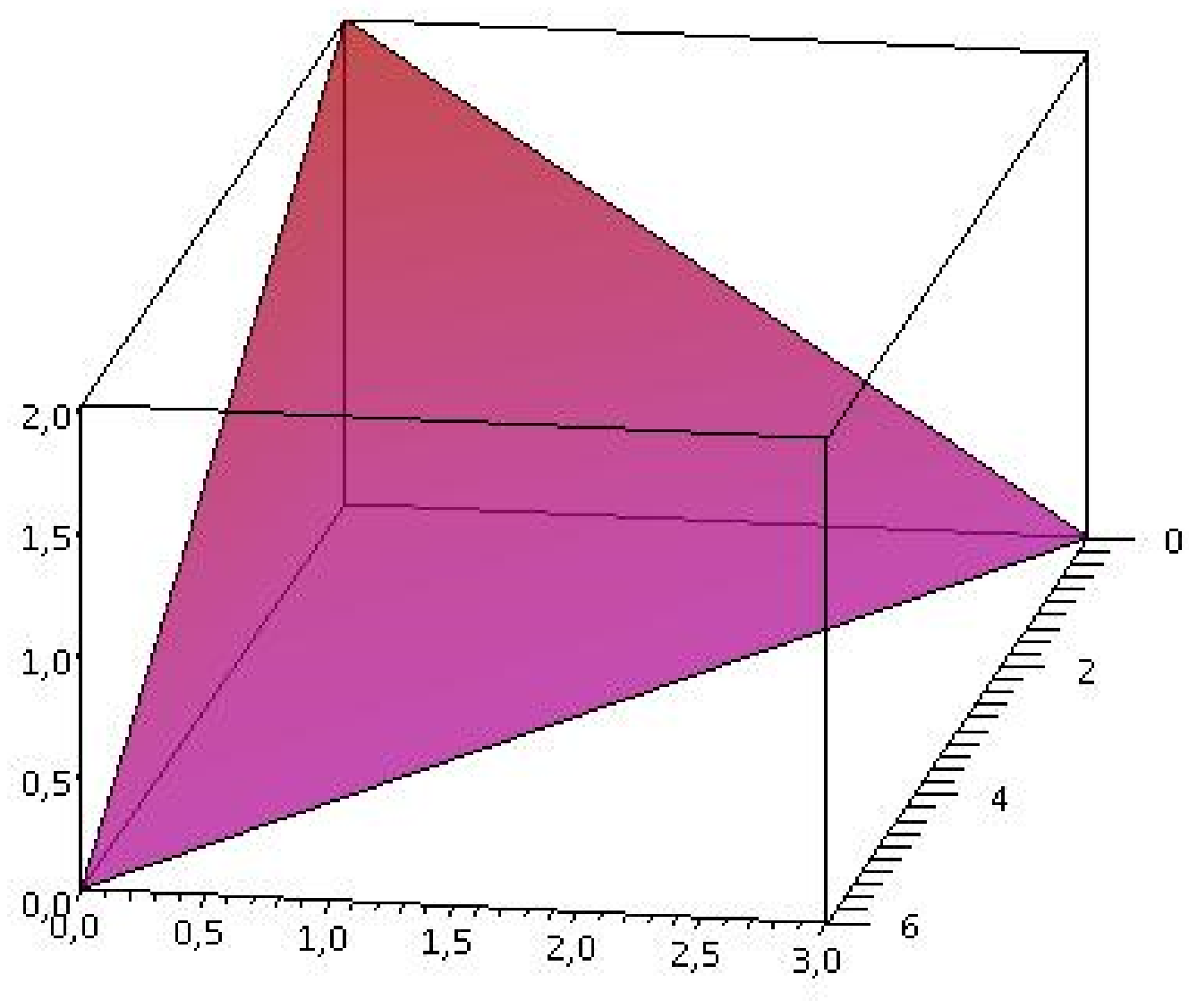} 
\end{minipage}

\caption{The maximal exceptional lattice simplices}
\end{figure}

\begin{figure}[ht]
\begin{minipage}[c]{.3\textwidth}
\centering\includegraphics[width=1.5in]{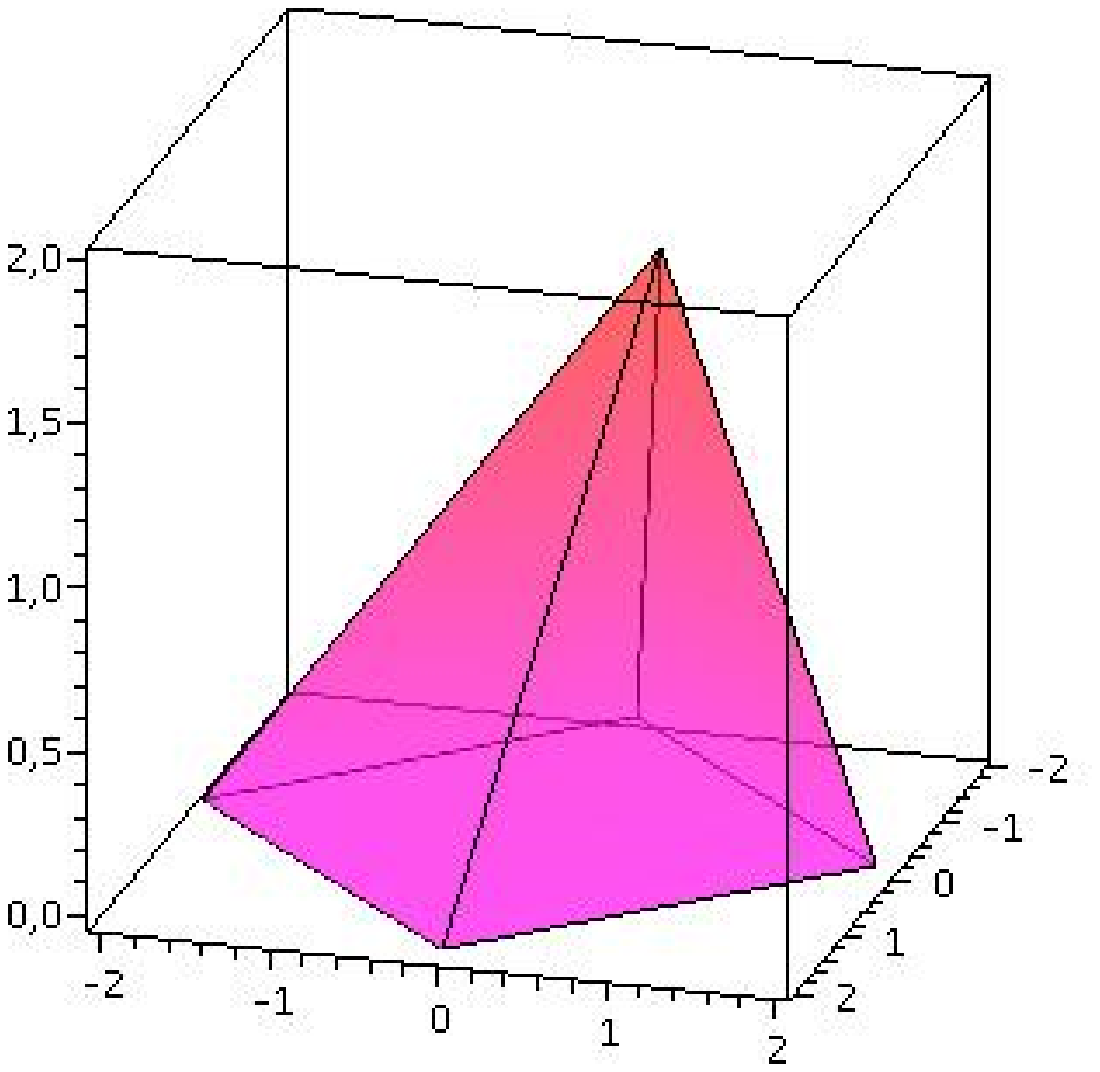} 
\end{minipage}
\begin{minipage}[c]{.3\textwidth}
\centering\includegraphics[width=1.5in]{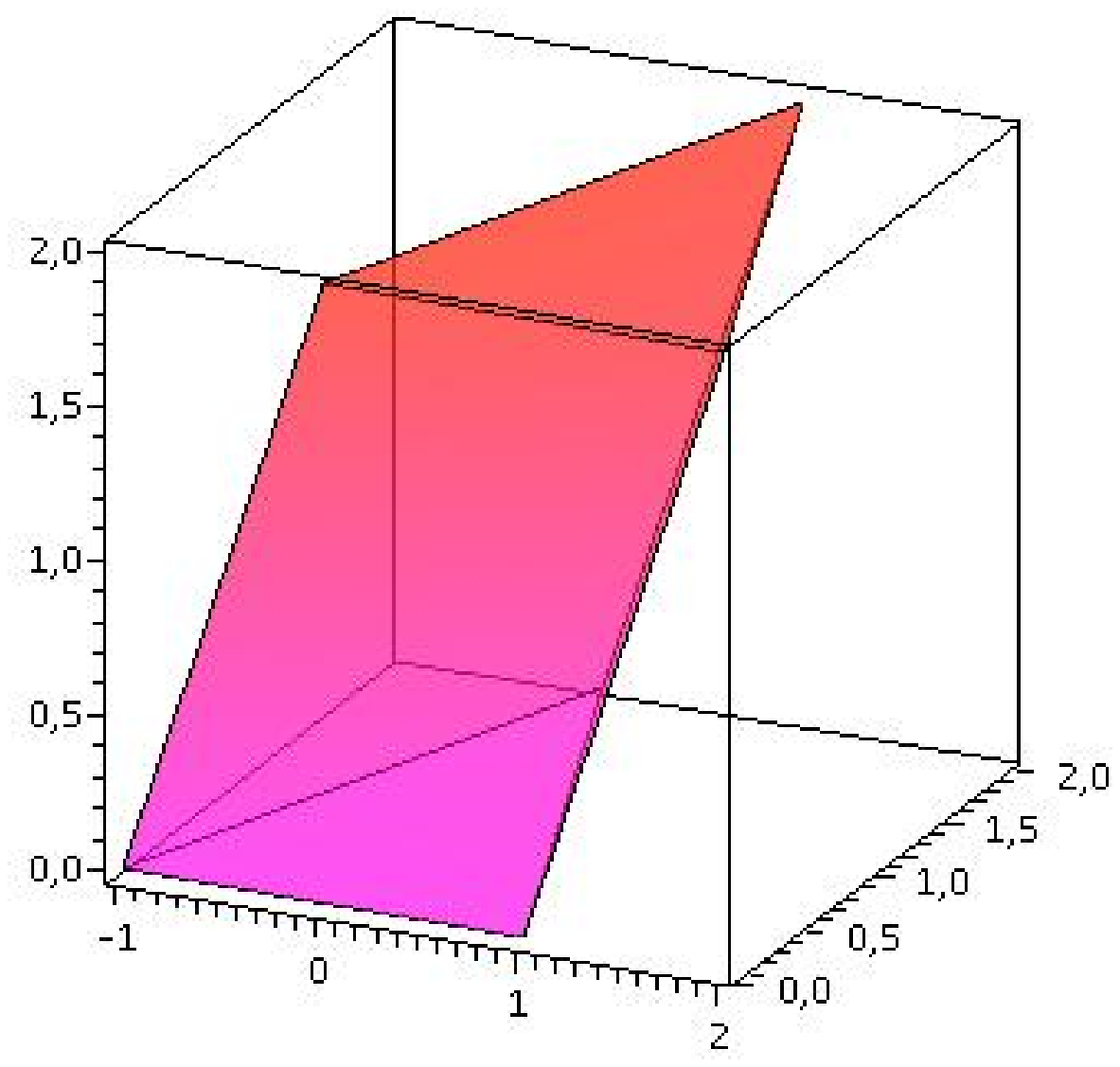} 
\end{minipage}
\begin{minipage}[c]{.3\textwidth}
\centering\includegraphics[width=1.5in]{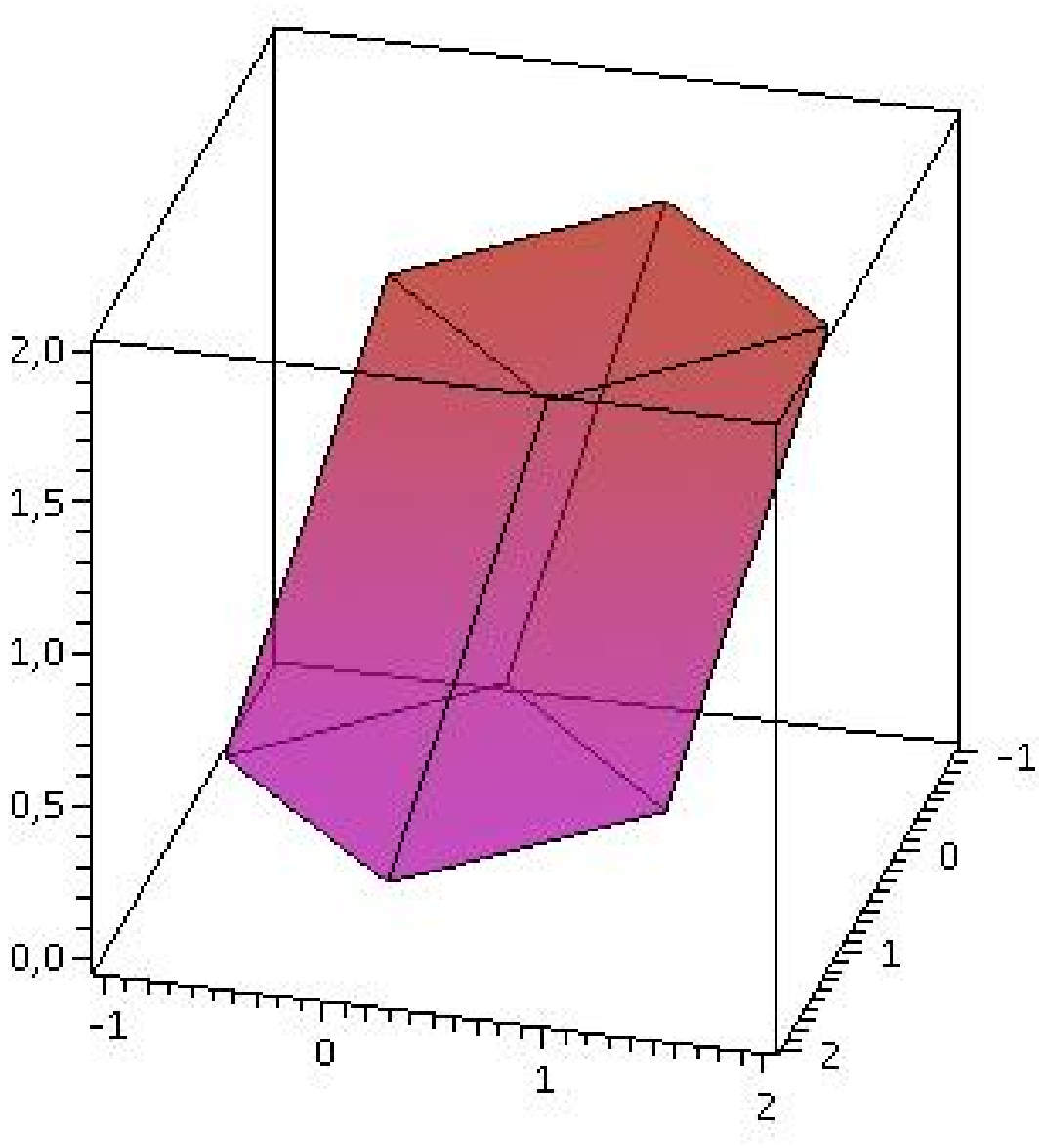} 
\end{minipage}
\caption{Some more maximal exceptional lattice polytopes}
\end{figure}

\bigskip

\begin{conj}
 Up to a finite number of exceptional lattice polytopes to every 
$n$-dimensional lattice polytope without interior lattice points, 
there is a projection mapping it onto an $m$-dimensional lattice 
polytope without interior lattice points, $m<n$.
\end{conj}

\begin{exam}
 Christian Haase and G\"unter M. Ziegler showed in \cite{HZ} that 
the only lattice points of the 4-dimensional lattice simplex 
$\conv(e_1,e_2,e_3,e_4,2e_1+2e_2+3e_3+(k-6)e_4)$ are its vertices, 
if and only if $\gcd(k,6)=1$. Moreover they say that the only lattice 
points of the 5-dimensional lattice simplex 
$\conv(0,e_1,e_2,e_3,e_1+e_2+e_3+6e_4,2e_1+3e_2+4e_3+9e_5)$ are its 
vertices. It is easy to see that there are projections mapping them 
onto $2\Delta_2$.
\end{exam}

Now we will show how you can get a smooth projective surface associated 
to a 3-dimensional lattice polytope:

\medskip

Given a lattice polytope $P\subset\R^3$, consider the normal fan 
$\Sigma(P)$ defined by the cones $$\sigma(\theta):=\{y\in(\Z^3)^*\ : \ 
\min_{x\in P}\langle x,y\rangle=\langle z,y\rangle \ 
\forall z\in\theta\},$$ whereas $\theta$ is a face of $P$.

\medskip

Denote by $\P_P$ the projective toric variety defined by this fan over a 
field $K$. Now consider the compactification $\S_P\subset \P_P$ of the 
surface $$\Big\{\sum_{v\in P\cap\Z^3}a_vX_1^{v_1}X_2^{v_2}X_3^{v_3}=0\Big\}
\subset(\overline{K}^*)^3\hookrightarrow\P_P$$ in $\P_P$, whereas 
$\overline{K}$  denotes the algebraic closure of $K$ and the coefficents 
$a_v\in K$ satisfy the following generic condition: $$\Big\{\sum_{v\in 
\theta\cap\Z^3}a_v X_1^{v_1}X_2^{v_2}X_3^{v_3}=0\Big\}\subset(
\overline{K}^*)^3$$ is smooth for every face $\theta\subset P$.

\medskip

A.G. Khovanskii showed in \cite{Kh78} that the number of interior lattice 
points of $P\subset\R^3$ is the same as the geometric genus of the surface
 $\S_P$. So by considering lattice polytopes $P\subset\R^3$ without interior
 lattice points we get -- after resolution of singularities -- some smooth 
projective surfaces of genus $0$, i.e. smooth projective surfaces which 
are ``nearly'' rational.

\medskip

In fact if $P$ is a Cayley polytope of two polygons, the corresponding 
surface will be rational. However if $P$ has a projection onto $2\Delta_2$ 
we will get a conic bundle, which -- in general -- will be not rational. 
Now consider the maximal exceptional lattice polytopes.

\medskip

$\S_{P_1}$ is the Godeaux surface $$\left\{\sum_{i=1}^4 x_i^5=0,\ x\in 
\P^3\right\}/\mu_5,$$ whereas $\zeta.x_i:=\zeta^i x_i$, $\zeta^5=1$. 

$\S_{P_5},\S_{P_7},\S_{P_8}$ and $\S_{P_9}$ are Enrique surfaces. $\S_{P_5}$ can be 
realized as
$$\S_{P_5}:=\left\{x_1^8+x_2^8+x_3^4+x_4^2=0,\ x\in \P^3(1,1,2,4)\right\}/\mu_2,$$ 
whereas $(-1).x:=(-x_1:x_2:x_3:x_4)$.

$\S_{P_8}$ is the closure of $$\left\{((u:v),(x:y:z))\in \P^1\times\P^2(1,1,2)\ :\ (u^2+v^2)(x^4y^4+x^4z^2+y^4z^2)=0\right\}/\mu_2,$$ whereas 
$(-1).((u:v),(x:y:z)):=((u:-v),(x:-y:-z))$.

$\S_{P_9}$ is the closure of \begin{eqnarray*}
\Big\{\left((x_0:x_1),(y_0:y_1),(z_0:z_1)\right)\in \P^1\times\P^1\times\P^1\ : \\ (x_0^2+x_1^2)(y_0^2+y_1^2)(z_0^2+z_1^2)=0\Big\}/\mu_2,\end{eqnarray*} whereas $(-1).(x,y,z):=\Big((x_0:-x_1),(y_0:-y_1),(z_0:-z_1)\Big).$

$\S_{P_3}$ is the cubic del Pezzo surface $$\left\{\sum_{i=1}^4 x_i^3=0,\ x\in 
\P^3\right\}.$$

$\S_{P_2}$ is the quasi-homogeneous nonic in weighted projective space 
$$\S_{P_2}:=\left\{x_1^9+x_2^9+x_3^9+x_4^3=0,\ x\in \P^3(1,1,1,3)\right\}/\mu_9,$$ 
whereas $\zeta.x:=(\zeta x_1:\zeta^4 x_2:x_3:\zeta^6 x_4)$, $\zeta^9=1$.

$\S_{P_4}$ is the quasi-homogeneous quartic in weighted projective space 
$$\left\{x_1^4+x_2^4+x_3^4+x_4^2=0,\ x\in \P^3(1,1,1,2)\right\}.$$

$\S_{P_6}$ is the quasi-homogeneous sextic in weighted projective space 
$$\left\{x_1^6+x_2^6+x_3^3+x_4^2=0,\ x\in \P^3(1,1,2,3)\right\}.$$

\bigskip

The degree of a polytope $P\subset\R^n$ is the biggest integer $d\in\N$ such that $$\Big|\Big((n+1-d)P\Big)^\circ\cap\Z^n\Big|>0.$$ Consequently, the degree of a 
lattice polytope $P\subset\R^3$ without interior lattice points is $0, 1$ or $2$. 
The lattice polytopes of degree smaller than $2$ are allready classified 
(cf. \cite{BN}), which generalizes the result of Arkinstall, Khovanskii, 
Koelman and Schicho:

\medskip

\begin{theorem}[Batyrev, Nill]
 \label{BN}
Let $P$ be an $n$-dimensional lattice polytope. If $\deg(P)=0$, then $P\cong\Delta_n$. 
If $\deg(P)=1$, then $P$ is an $(n-2)$-fold lattice pyramid over $2\Delta_2$ or a 
Lawrence polytope, i.e. a Cayley polytope of $n$ intervals.
\end{theorem}

\begin{coro}
Let $P\subset\R^n$, $n\geq3$, be a lattice polytope and $\deg(P)\leq1$. Then is $P$ 
a Cayley polytope.
\end{coro}

Let $P\subset\R^3$ be a lattice polytope without interior lattice points. Then 
$\deg(P)\leq2$. Thus it remains to consider the case $\deg(P)=2$.\\
In the following we will generalize Theorem \ref{Howe} step by step in order to 
prove Theorem \ref{Main} in the end.

\bigskip
\noindent {\bf Acknowledgments:} The author would like to thank Victor Batyrev for 
discussions and joint work on this subject.

\bigskip
\section{White Lattice Simplices}
In the following we denote by $\Vol(P)=n!\vol(P)$ the normalized volume of an 
$n$-dimensional lattice polytope.

\medskip

If $F\subset\R^2$ is a lattice polygon without interior lattice points, we 
derive $F\cong 2\Delta_2$ or $F$ is a Lawrence polytope from section \ref{1}.

\medskip

We describe a 3-dimensional lattice polytope as white, if all its lattice points 
are on its edges. In \cite{Whi} G.K. White also proved the following generalization 
of Theorem \ref{Howe}, using inequalities of Gaussian brackets. We give a new proof:
\begin{prop} 
\label{White}
Let $P\subset\R^3$ be a white lattice simplex with $F\not\cong2\Delta_2$ for every 
facet $F\subset P$ of $P$. Then is $P$ a Cayley polytope.
\end{prop} 
\begin{proof}
As every facet of $P$ is a Lawrence polygon, there are at most two edges of $P$ of 
length bigger than $1$, i.e. edges with more than two lattice points. By Theorem 
\ref{Howe} we may assume that there is at least one edge of length bigger than 1.

\begin{enumerate} 
%
%
\item[Case 1] \textbf{There is exactly one edge of length $l_1>1$.}\\
Denote the vertices of $P$ by $a,b,c,d\in\Z^3$ and let $f\in\conv(a,b)^\circ\cap\Z^3$ 
be an interior point of the edge $\conv(a,b)$.

\medskip

\begin{enumerate}
\item[Case A] \textbf{$l_1=2$}

\medskip

The lattice simplices $\Delta:=\conv(a,c,d,f)$ and $\Delta':=\conv(b,c,d,f)$ 
are Cayley polytopes by Theorem \ref{Howe}. So $\Delta$ is streched between 
parallel planes $E_1,E_2$ having distance $1$, and $\Delta'$ is streched 
between $E_1',E_2'$ having also distance $1$.

\medskip

If $a$ and $f$ are both contained in $E_1$ or $E_2$, then it will follow 
by linearity that $b$ is also contained in the same plane. Hence $P$ is 
a Cayley polytope streched between $E_1$ and $E_2$. The same argument 
holds if $b$ and $f$ are both contained in $E_1'$ or $E_2'$.

\medskip

If $\Delta$ is a pyramid, then $\Delta\cong\Delta_3$ will follow. So 
$P=\Delta\cup\Delta'\cong\conv(-e_1,$ $e_1,e_2,e_3)$ is a Cayley polytope. 
The same argument holds for $\Delta'$. So assume in the following that 
$\Delta$ and $\Delta'$ are not pyramids.

\medskip

Let $f=0, c=e_1,d=e_2,b=-a$ and $0\leq a_1,a_2<a_3$. Notice that 
$\Vol(\Delta)=\Vol(\Delta')=a_3$. We may assume that $\Delta$ is a Cayley 
polytope of $\Delta_1=\conv(0,e_1)$ and $\Delta_2=\conv(e_2,a)$, i.e. 
$\Vol\Big(\conv(0,e_1,a-e_2)\Big)$ $=\Vol(\Delta)=a_3$. Hence $\gcd(a_2-1,a_3)=a_3,$ 
which implies $a_3|(a_2-1)$. So $a_2=1$. It remains to consider the following 
two cases:

\medskip

\begin{enumerate}
\item[Case a:] $\Delta'$ is the Cayley polytope of $\Delta'_1=\conv(0,e_1)$ 
and $\Delta'_2=\conv(e_2,-a)$.
Here is $\Vol\Big(\conv(0,e_1,a+e_2)\Big)=\Vol(\Delta')=a_3$. Hence 
$\gcd(a_2+1,$ $a_3)=a_3,$ which implies $a_3|(a_2+1)=2$. As $a_3>a_2=1$ 
we derive $a_3=2$. Then $a_1\in\{0,1\}$.\\
If $a_1=0$, then $P\subset\{0\leq x_1\leq 1\}$ will be a Cayley polytope. 
If $a_1=1$, then $e_1$ and $e_2$ will have distance 1 to the plane 
$\lin(a,e_1-e_2)$, which includes $0,a,-a$. So again $P$ will be a Cayley polytope.

\medskip

\item[Case b:] $\Delta'$ is the Cayley polytope of $\Delta'_1=\conv(0,e_2)$ 
and $\Delta'_2=\conv(e_1,-a)$.
Here is $\Vol\Big(\conv(0,e_2,a+e_1)\Big)=\Vol(\Delta')=a_3$. Hence 
$\gcd(a_1+1,$ $a_3)=a_3,$ which implies $a_3|(a_1+1)$. Then $e_1$ and $e_2$ 
have distance 1 to the plane $\lin(a,e_1-e_2)$, which includes $0,a,-a$. 
Thus $P$ is a Cayley polytope.
\end{enumerate}

%
%
\medskip

\item[Case B] \textbf{$l_1>2$}

\medskip

By induction $\Delta$ and $\Delta'$ are both Cayley polytopes, whereas 
we adopt the same notation as before. Without loss of generality let 
$|\conv(a,f)^\circ$ $\cap\Z^3|>0$. Then $a$ and $f$ -- and thus also 
$b$ -- are contained in $E_1$ or $E_2$. Hence, $P$ is a Cayley 
polytope strechted between $E_1$ and $E_2$.
\end{enumerate}
%
%
\medskip

\item[Case 2] \textbf{There are two edges of lengths $l_1>1$, $l_2>1$.}

\medskip

Let $l_1+1=|\conv(a,b)\cap\Z^3|$ and $l_2+1=|\conv(c,d)\cap\Z^3|.$ There 
is a lattice point $f\in\conv(a,b)^\circ\cap\Z^3$ such that 
$|\conv(b,f)^\circ\cap\Z^3|=0$.

\medskip

By induction, $\Delta^{(1)}:=\conv(a,c,d,f)$ and $\Delta^{(2)}:=\conv(b,c,d,f)$ 
are Cayley polytopes, i.e. there is a number $z^{(i)}\in\Z$ and a linear 
functional $y^{(i)}$ $\in(\Z^3)^*$ in such a way, that $\Delta^{(i)}$ is 
contained in $W^{(i)}:=\{x\in\R^3: \ z^{(i)}\leq \langle x,y^{(i)}\rangle 
\leq z^{(i)}+1\}$, $i\in\{1,2\}$.

\medskip

The lattice polygon $\conv(c,d,f)$ is a Lawrence polygon. So every facet 
of $\Delta^{(i)}, \ i\in\{1,2\}$ is a Lawrence polygon.

\medskip

If $a$ and $f$ are both contained in the same boundary plane of $W^{(1)}$ 
then $b$ will be in it too. In this case $P\subset W^{(1)}$ is a Cayley 
polytope.\\ 
So let $a$ and $f$ be in different boundary planes of $W^{(1)}$. Thus 
$l_1=2$.\\
As $l_2>1$ we notice that $c$ and $d$ are in the same boundary plane of 
$W^{(i)}$.

\medskip

Hence $\Delta^{(i)}$ is a pyramid for $i\in\{1,2\}$ and thus a Lawrence 
polytope. We may choose $\Delta^{(1)}=\conv(0,e_1,l_2e_2,e_1+e_3)$. So 
we see that $P\subset\{0\leq X_1\leq1\}$ is a Cayley polytope.\qedhere
\end{enumerate}
\end{proof}

\medskip

\begin{prop}
\label{Simplex} 
Let $P\subset\R^3$ be a white lattice simplex. Then is $P$ a Cayley 
polytope or $P\cong2\Delta_3$.
\end{prop} 
\begin{proof} Because of Proposition \ref{White} we may assume 
$P=\conv(0,2e_1,2e_2,d)$ with $d\in\Z^3, d_3>1$. 
\medskip

If there is a facet $F\cong2\Delta_2$ of $P$, $F\not=P\cap X_3^\perp$, 
then every facet of $P$ will be of this kind and $\frac{1}{2}P$ will be 
a lattice polytope of degree $0$ or $1$ having only unimodular facets. 
So $P\cong 2\Delta_3$.

\medskip

Now let every facet $F\not=P\cap X_3^\perp$ be a Lawrence polygon. As $P$ 
is a simplex and $v:=(v_1,v_2,1)\not\in P \ \forall v_1,v_2\in\Z$, the line 
$g$ through $d$ and $v$ intersects with $P$ only in $d$. In particular 
$g\cap X_3^\perp\not\in 2\Delta_2$. Consequently $M\cap2\Delta_2=\emptyset$, 
whereas $$M=\Big\{v_1\cdot\frac{d_3}{d_3-1}-w_1,\ v_2\cdot\frac{d_3}{d_3-1}-w_2: 
\ v_1,v_2\in\Z \Big\}$$ with some $w_1,w_2\in\Z$. But this is a contradiction 
to the fact that $\emptyset\not=M\cap\{0\leq X_1,X_2$ $\leq1\}\subset 2\Delta_2$.
 So, $d_3=1$ and $P$ is a Cayley polytope.
\end{proof}
\bigskip
\section{White Lattice Polytopes}
\begin{prop}
\label{Pyramide}
Let $F$ be a lattice polygon, $d\in\Z^3$ and $P=\conv(F,d)\subset \R^3$ 
be a white lattice polytope. Then is $P$ a Cayley polytope.
\end{prop}
\begin{proof}
 As $|F^\circ\cap\Z^3|=0$, we may suppose $F$ is a Lawrence polytope 
but not a simplex by Theorem \ref{BN} and Proposition \ref{Simplex}. 
So let $F$ be the Cayley polytope of two parallel edges $f^{(1)}$ and 
$f^{(2)}$. Assume that $X_3(f^{(1)})=X_3(f^{(2)})=0$ and 
$X_1(f^{(1)})=0, X_1(f^{(2)})=1$.

\medskip

Choose two intervals $\Delta_1\cong q^{(i)}\subset f^{(i)}$, $i\in\{1,2\}$, 
and consider the lattice subpolytope $Q:=\conv(q^{(1)},q^{(2)},d)\subset P$. 
Let $0\leq d_1,d_2<d_3$ and suppose $d_3>1$.

\medskip

Assume that the only lattice points of $Q$ are its vertices. Then is $Q$ a 
Cayley polytope by Theorem \ref{Howe}. But as all the facets of $Q$ have 
no interior lattice points we conclude that every subsimplex is unimodular 
and hence $d_3=1$, which is a contradiction. So we may assume that 
$\conv(0,d)^\circ\not=\emptyset$, $q_1=\conv(0, e_2)$ and 
$q_2=\conv(e_1,e_1+e_2)$. We will show that $d_1=0$, which will imply that 
$P\subset\{0\leq X_1\leq1\}.$

\medskip

Subdivide $Q$ into $\Delta^{(1)}:=\conv(0,e_1,e_1+e_2,d)$ and 
$\Delta^{(2)}:=\conv(0,e_2,e_1+e_2,d)$. By Proposition \ref{Simplex} 
they are both Cayley polytopes of normalized volume $d_3$. 

\medskip

If $\Delta^{(1)}$ is a pyramid, then it will be a Lawrence simplex and 
hence $d_1=d_2=0$.

\medskip

So let $\Delta^{(1)}$ be a Cayley polytope of $\conv(0,d)$ and 
$\conv(e_1,e_1+e_2)$ resp. $\conv(e_2,e_1+e_2)$. Then is 
$\Vol\Big(\conv(0,e_2,d)\Big)=\Vol(\Delta)=d_3$. Thus $\gcd(d_1,d_3)$ $=d_3$, 
which implies $d_3|d_1$. Hence $d_1=0$. 
\end{proof}

\medskip

\begin{prop} 
\label{Circuit}
Let $P\subset \R^3$ be a white lattice polytope with at most $5$ vertices. 
Then is $P$ a Cayley polytope or $P\hookrightarrow3\Delta_3$ and 
$(2\Delta_2)^\circ\hookrightarrow P^\circ$.
\end{prop}
\begin{proof}
By the Propositions \ref{Simplex} and \ref{Pyramide} we may assume that $P$ 
is a circuit, i.e. the convex hull of a 2-dimensional lattice simplex $\Delta$ 
and a 1-dimensional lattice polytope $\conv(a,b)$, each without interior 
lattice points, such that $\conv(a,b)^\circ,\Delta^\circ\subset P^\circ$. 
By Theorem \ref{Howe} we may moreover assume that there is an edge of $P$ 
of length $n\geq2$. Consider the following two cases:

\medskip

\begin{enumerate}
 \item[Case A:] $\Delta\cong\Delta_2.$

\medskip

There are lattice points $c,d\in P\cap\Z^3$ such that $\conv(\Delta,c,d)$ 
is a circuit whose only lattice points are its vertices. By Theorem 
\ref{Howe} we may suppose that $\conv(\Delta,c,d)=\conv(0,e_1,e_2,-e_3,f)$ 
with $\Delta=\conv(0,e_1,e_2)$, $f=e_1+xe_2+ye_3$ and $x,y\in\N$. As $P$ is 
a circuit we see moreover $x>0$ and $1+x<y$, in particular $y>1$. In order 
to get $P$ back, it remains to elongate $f$ or $-e_3$ from $\Delta$. If we 
elongate $f$ from $e_1$, then $P\subset\{0\leq X_1\leq1\}$ will be a Cayley 
polytope. The same will be if we elongate $-e_3$ from $0$ or from $e_2$. 
We will show now by assuming the contrary that in the remaining three cases 
you cannot reach any subpolytope of $P$:

\medskip
\begin{enumerate}
 \item [Case 1:] Elongate $f$ from $0$.

\medskip
Here we get the circuit $\conv(0,e_1,e_2,-e_3,2f)\subset P$. By 
Proposition \ref{Simplex} is the subsimplex $\conv(0,e_1,e_2,2f)$ 
a Cayley polytope.

\medskip
If it was a pyramid, then it would follow from Theorem \ref{BN} that 
$f_3|f_1$ and $f_3|f_2$, which would be a contradiction.
 \medskip

Thus $\conv(0,e_1,e_2,2f)$ is the Cayley polytope of $\conv(0,2f)$ and 
$\conv(e_1,e_2)$. So $\Vol\Big(\conv(0,2f,e_1-e_2)\Big)=
\Vol\Big(\conv(0,e_1,e_2,2f)\Big)$ $=2y$ and consequently $2x+2=2y$, which 
is a contradiction to the fact that $1+x<y$.

\medskip

 \item [Case 2:] Elongate $f$ from $e_2$.

\medskip
Here we get the circuit $\conv(0,e_1,e_2,-e_3,2f-e_2)\subset P$. As before 
we receive $2y=\Vol\Big(\conv(0,2f-e_2,e_1-e_2)\Big)$, which implies $2y|(2x+1)$, 
which is a contradiction.

\medskip

 \item [Case 3:] Elongate $-e_3$ from $e_1$.

\medskip

Here we get the cirucit $\conv(0,e_1,e_2,-e_1-2e_3,f)$. So $y=1$, which is 
again a contradiction.
\end{enumerate}

\medskip

 \item[Case B:] There is an edge of $\Delta$ of length $n\geq2$.

\medskip

By Theorem \ref{BN} we may assume $\Delta=\conv(0,ne_1,\delta e_2)$ with 
$\delta\in\{1,2\}, a_3>0>b_3$. Subdivide $P$ into $\Delta':=P\cap\{X_3\geq0\}$ 
and $\Delta'':=P\cap\{X_3\leq0\}$. By Proposition \ref{Simplex} they are 
Cayley polytopes or equivalent to $2\Delta_3$.

\medskip

If $\Delta'\cong2\Delta_3$ then $n=\delta=2$ and without loss of 
generality $a=2e_3$. Then $\Delta''\cong2\Delta_3$ or $b_3=-1$. 
The first case implies $\deg\Big(\frac{P}{2}\Big)\leq1$ and 
$\frac{P}{2}$ has only unimodular facets, which is not possible 
for a circuit. In the second case we get $P\hookrightarrow 3\Delta_3.$

\medskip

So let in the following $\Delta'$ and $\Delta''$ both be Cayley polytopes.

\medskip

\begin{enumerate}
 \item[Case 1:] $\Delta'$ and $\Delta''$ are pyramids over $P\cap X_3^\perp$.

\medskip

Here we get $a_3=-b_3=1.$ So we may assume $a=e_3$. If $\delta=n=2$, then 
$b=e_1+e_2-e_3$, $b=2e_1+e_2-e_3$ or $b=e_1+2e_2-e_3$, i.e. 
$P\hookrightarrow 3\Delta_3$.\\
 Now let $\delta=1$. As $P$ is a circuit, we get $b_2=1$. Then 
$P\subset\{0\leq X_2\leq1\}$ is a Cayley polytope.

\medskip

 \item[Case 2:] $\Delta'$ and $\Delta''$ are pyramids.

\medskip

Let $\Delta'$ be a pyramid over $\conv(0,ne_1,a)$. Then $\delta=1$ and 
$\conv(0,ne_1,$ $a)$ is a Lawrence polytope. Thus is $\Delta'$ a Lawrence 
polytope and this case is reduced to case 1.

\medskip

 \item[Case 3:] Exactly one of the simplices $\Delta'$ and $\Delta''$ 
is a pyramid.

\medskip

Without loss of generality let $\Delta''$ be a pyramid and $\Delta'$ 
be the Cayley polytope of $\Delta'_1=\conv(0,ne_1)$ and 
$\Delta'_2=\conv(e_2,a)$. Moreover let $0\leq a_2<a_3>1$. Like in 
case 1 and 2 we derive $b_3=-1$, $\delta=1$ and $\Vol(\Delta')=na_3$.\\
Also $\Vol\Big(\conv(0,ne_1,a-e_2)\Big)=\Vol(\Delta')$ and hence 
$\gcd\Big(na_3,n(a_2-1)\Big)=na_3$, which implies $a_3|(a_2-1)$. So 
$a_2=1$. As $P$ is a circuit, we conclude $b_2=0$. Hence 
$P\subset\{0\leq X_2\leq1\}$ is a Cayley polytope.

\medskip

 \item[Case 4:] Neither $\Delta'$ nor $\Delta''$ is a pyramid.

\medskip

Like in case 3 we may assume $a_2=1$ and $b_3|(b_2-1)$. Consider the 
projection of $P$ to $X_1^\perp$. As $\Delta'$ is not a pyramid, we 
notice $a_3\geq2$. Then we get $b_2\leq0$ and $b_3-a_3b_2<0$, because 
$P$ is a circuit. Now $-b_3|(1-b_2)$, which is smaller than 
$1+\frac{-b_3}{a_3}$. This implies $b_3=-1$, which is a contradiction.
\qedhere
\end{enumerate}
\end{enumerate}
\end{proof}
\begin{prop}
 \label{empty}
Let $P\subset\R^3$ be a white lattice polytope. Then is $P$ a Cayley 
polytope or $\Vol(P)\leq C^{(\ref{empty})}$.
\end{prop}
\begin{proof}
If $2\Delta_2\hookrightarrow P$, then by Proposition \ref{Simplex} 
every lattice point of $P$ will have at most distance $2$ from 
$2\Delta_2$. This bounds $P$. Thus let $2\Delta_2\not\hookrightarrow P$. 
Then will every facet of $P$ be a Lawrence polytope.

\medskip 

Let $\conv(0,ne_3)$, $n\in\N$ be the longest edge of $P$, with 
adjacent facets $F_1,F_2$. By Howe's Theorem \ref{Howe} and the 
Propositions \ref{Simplex}, \ref{Circuit} we may suppose $n\geq2$ 
and that $P$ has more than $5$ vertices. Choose some vertices 
$a\in F_1$, $b\in F_2$ such that $\Delta:=\conv(0,a,b,ne_3)$ is a 
3-dimensional lattice simplex.

\medskip

If there is no lattice point $p\in P\backslash\{F_1\cup F_2\}
\cap\Z^3$, then $P$ contains a pyramid and will also be a Cayley 
polytope by Proposition \ref{Pyramide}. So let $p\in P
\backslash\{F_1\cup F_2\}\cap\Z^3$.

\medskip

By Proposition \ref{Circuit} we see that $\conv(0,ne_3,a,b,p)$ 
is a Cayley polytope. So there are possibilities: 

\medskip

\begin{enumerate}
\item[$A$] The distance of $p$ and $b$ to $F_1$ is 1.
\item[$B$] The distance of $p$ and $a$ to $F_2$ is 1.
\item[$C$] A plane through $a,b$ and $p$ has distance $1$ to $0$ 
and $ne_3$.
\end{enumerate}

\medskip

  If there is a lattice point $q\in P\cap\Z^3$ of type $A$, not 
of type $B$ and a lattice point $q'\in P\cap\Z^3$ of type $B$, 
not of type $A$, then we may assume $a=e_1$ and $b=e_2$. But then 
is $\conv(0,ne_3,q,q')$ not a Cayley polytope, which is a 
contradiction to Proposition \ref{Simplex}.

\medskip

  If there is a lattice point $q\in P\cap\Z^3$ of type $A$ or type $B$ 
and a lattice point $q'\in P\cap\Z^3$ of type $C$, then we may assume 
$a=e_1$ and $b=e_2$. But then there is no possible position for $q'$.

\medskip

So it is only possible to have further lattice points $p\in P\backslash
\{F_1\cup F_2\}\cap\Z^3$ of either type $A$ or type $B$ or type $C$. 
This means that $P$ is a Cayley polytope.\qedhere
\end{proof}
\bigskip
\section{Lattice Polytopes Without Interior Lattice Points}
In the following it remains to consider lattice polytopes having at 
least one facet with an interior lattice point in order to prove 
Theorem \ref{Main}.

\begin{rem} 
\label{Pick}
Let $F\subset \R^2$ be a lattice polygon with $\Vol(F)\leq 4$ and 
$F^\circ\cap\Z^2\not=\emptyset$. Pick's formula states $\Vol(F)=
|F\cap\Z^2|+|F^\circ\cap\Z^2|-2$. So we derive 
$|F^\circ\cap\Z^2|=1$ and $F$ is one of the following:
$$\{\conv(e_1,e_2,-e_1-e_2),\conv(\pm e_1,2e_2),
\conv(\pm e_1,\pm e_2),\conv(e_1,e_2,\pm(e_1+e_2))\}$$
\end{rem}
\begin{figure}[ht]
\begin{minipage}[c]{0.2\textwidth}
\centering \includegraphics[width=0.8in]{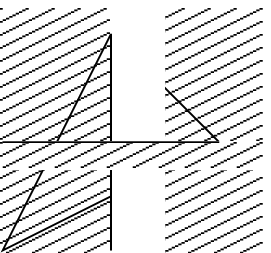}
\end{minipage}%
\begin{minipage}[c]{0.2\textwidth}
\centering \includegraphics[width=.8in]{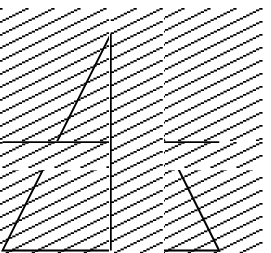}
\end{minipage}%
\begin{minipage}[c]{0.2\textwidth}
\centering\includegraphics[width=.8in]{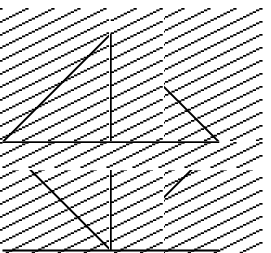}
\end{minipage}
\begin{minipage}[c]{0.2\textwidth}
\centering\includegraphics[width=.8in]{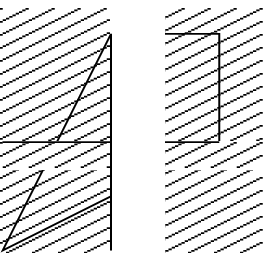}
\end{minipage}
\caption{The 4 smallest lattice polygons with an interior 
lattice point}
\end{figure}

\begin{prop} 
\label{Hensley}
Let $P=\conv(F,d)\subset\R^3$ be a lattice simplex without interior 
lattice points and $|F^\circ\cap\Z^3|=1,\  d\in\Z^3,d_3>0.$ Let every 
further facet of $P$ be without interior lattice points. Then is 
$d_3$ bounded by a constant $C^{(\ref{Hensley})}>0$.
\end{prop}
\begin{rem}
\label{d<7}
 In fact, it is possible to show that 
%
%
$C^{(\ref{Hensley})}=6.$
\end{rem}
\begin{proof}
 There is a lattice point $p\in \Z^3\backslash P$ such that 
$\conv(p,P)=P\cup\conv(p,F)$ and $|\conv(p,P)^\circ\cap\Z^3|=1$. 
From Theorem \ref{LZ} it follows that $d_3$ is bounded.
\end{proof}
 It is easy to generalize this idea in order to prove the following:
\begin{prop}
Let $F\subset\R^n$ be an $(n-1)$-dimensional lattice poytope with 
$i>0$ interior lattice points and $d\in\Z^n$ such that $P=\conv(F,d)$ 
is an $n$-dimensional lattice polytope without interior lattice points. 
Then is the volume of $P$ bounded by a constant depending only on $n$ 
and $i$.
\end{prop}

\bigskip

\begin{prop}
\label{inner}
Let $P\subset\R^3$ be a lattice polytope and $F\subset P$ be a facet of 
$P$ with $|F^\circ\cap\Z^3|>0$. Then is $P$ a Cayley polytope, $P$ can 
be projected to the double unimodular 2-simplex or 
$\Vol(P)<C^{(\ref{inner})}$.
\end{prop}
\begin{proof} Let without loss of generality $F=P\cap X_3^\perp$ and 
$X_3|_P\geq0$. Define $s:=P\cap(X_3-1)^\perp$. Assume $P$ is not 
contained in $\{0\leq X_3\leq 1\}$, i.e. $s^\circ\subset P^\circ$.

\medskip

 There is a projection $\pi:\Z^3\rightarrow \Z^2$ mapping $s$ onto an 
interval $I$ of length $d=\vol(I)\geq1$ such that $d\leq f:=\max_{x,y\in s}
\abs{\langle v,x-y\rangle}$, whereas $v$ is a primitive normal vector of $I$.

\medskip

If $d=1$ we will derive $\Vol(P)$ is bounded or $\pi(P)\cong2\Delta_2$, 
because there is a lattice point $p\in P\cap\Z^3$ satisfying $p_3>1$ and 
$\pi(P)$ is a lattice polytope. So let $d>1$. By $d\leq f$ and 
$s^\circ\cap\Z^3=\emptyset$ we receive after some computation $d\leq2$.\\
We will distinguish the following three cases:

\medskip

\begin{enumerate}
  \item[Case 1:] $1<d\leq\frac{3}{2}$.

\medskip
As $\pi(P)$ is a lattice polytope and there is a lattice point 
$p\in P\cap\Z^3$ with $p_3>1$, we see $p_3\leq 4$ $\forall p\in P\cap\Z^3$ 
and $d\in\{\frac{4}{3},\frac{3}{2}\}.$ By calculating we receive hence 
$\Vol(s)\leq\frac{16}{3}$. Therefore is $P$ contained in the truncated 
cone spanned by $F$ and $s$ and contained in $\{0\leq X_3\leq 4\}.$ This 
bounds $\Vol(P)$.

\medskip

 \item[Case 2:] $\frac{3}{2}<d\leq2$, $\Vol(F)>4$.

\medskip
By computation we receive $\Vol(s)\leq\frac{9}{2}$. As $\Vol(F)$ $\geq5$ 
we conclude $p_3<20$ $\forall p\in P\cap\Z^3$. Therefore $P$ is contained 
in the truncated cone spanned by $F$ and $s$ and contained in 
$\{0\leq X_3\leq 19\}.$ This bounds $\Vol(P)$.

 \medskip

 \item[Case 3:] $\frac{3}{2}<d\leq2$, $\Vol(F)\leq4$.

\medskip
Again $\Vol(s)\leq \frac{9}{2}$. By Remark \ref{Pick}, $F$ is equivalent 
to one of the 4 polygons from Remark \ref{Pick}. Moreover we may assume 
that every facet of $P$ with interior lattice points is one of them. Let 
$p\in P\cap\Z^3$ with $p_3$ maximal and consider the lattice polytope 
$\Delta_p:=\conv(F,p)$.

\medskip 

If $F$ is the only facet of $\Delta_p$ with an interior lattice point, 
we will derive $p_3\leq6$ by Remark \ref{d<7}.

\medskip 

Else let there be a facet $F'\not=F$ of $\Delta_p$ with an interior lattice 
point $q\in F'^\circ\cap\Z^3$. As $P$ has no interior lattice points, we 
derive $F'\subset\partial P$. Thus $F'$ is one of the four lattice polygons 
from Remark \ref{Pick}. As $q_3\leq6$ again, we receive $p_3\leq 3q_3=18$. 

\medskip 

Therefore $P$ is contained in the truncated cone spanned by $F$ and $s$ and 
contained in $\{0\leq X_3\leq 18\}.$ This bounds $\Vol(P)$.\qedhere
\end{enumerate}
\end{proof}

\bigskip

\begin{proof}[Proof of Theorem \ref{Main}]
If a lattice polytope $P\subset\R^3$ without interior lattice points is 
neither a Cayley polytope nor can be projected to the double unimodular 
2-simplex, it will follow from the Propositions \ref{empty} and \ref{inner} 
that $\Vol(P)\leq \max\{C^{(\ref{empty})},C^{(\ref{inner})}\}$. We conclude 
that there is only a finite number of such exceptional polytopes by 
Theorem \ref{LZ}.
\end{proof}
\bigskip
\bibliographystyle{amsalpha}

\end{document}